\documentclass[reqno]{amsart}
\usepackage[utf8]{inputenc}
\usepackage{amsmath}
\usepackage{amssymb}
\usepackage{xcolor}
\usepackage{graphicx}
\usepackage{float}
\usepackage{cite}
\newtheorem{theorem}{Theorem}
\newtheorem{theoremlett}{Theorem}

\newtheorem{corollary}[theorem]{Corollary}
\newtheorem{remark}[theorem]{Remark}

\newcommand{\Tr}{\mathrm{Tr}}
\newcommand{\C}{\mathbb{C}}

\newcommand{\ler}[1]{\left( #1 \right)}

\newcommand{\abs}[1]{\left| #1 \right|}

\newcommand{\be}{\begin{equation}}
\newcommand{\ee}{\end{equation}}
\newcommand{\ba}{\begin{array}}
\newcommand{\ea}{\end{array}}
\newcommand{\emp}{\mathrm{m}_p}
\newcommand{\eps}{\varepsilon}
\newcommand{\Oeps}{\underline{\underline{\mathrm{O}}}\left(\varepsilon^2\right)}
\newcommand{\kif}{\underline{\underline{\mathrm{O}}}\left(\varepsilon^3\right)}

\newcommand{\fel}{\frac{1}{2}}
\newcommand{\cA}{\mathcal{A}}

\newcommand{\tr}{\mathrm{tr}}

\title{Preservers of the $p$-power and the Wasserstein means on $2 \times 2$ matrices}

\author[Rich\'ard Simon]{Rich\'ard Simon}
\address{Rich\'ard Simon, Department of Analysis, Institute of Mathematics, Budapest University of Technology and Economics, Műegyetem rkp. 3., H-1111 Budapest, Hungary.}
\email{sr570063@gmail.com; simonr@math.bme.hu}

\author[D\'aniel Virosztek]{D\'aniel Virosztek}
\address{D\'aniel Virosztek, Alfr\'ed R\'enyi Institute of Mathematics\\ Re\'altanoda u. 13-15.\\Budapest H-1053\\ Hungary}
\email{virosztek.daniel@renyi.hu}

\date{}



\thanks{Virosztek was supported by the Momentum Program of the Hungarian Academy of Sciences (grant no. LP2021-15/2021) and partially supported by the ERC Synergy Grant No. 810115.}

\begin{document}

\maketitle

\begin{abstract}
In one of his recent papers \cite{ML1}, Moln\'ar showed that if $\cA$ is a von Neumann algebra without $I_1, I_2$-type direct summands, then any function from the positive definite cone of $\cA$ to the positive real numbers preserving the Kubo-Ando power mean, for some $0 \neq p \in (-1,1)$ is necessarily constant. It was shown in that paper, that $I_1$-type algebras admit nontrivial $p$-power mean preserving functionals, and it was conjectured, that $I_2$-type algebras admit only constant $p$-power mean preserving functionals. We  confirm the latter. A similar result occurred in another recent paper of Moln\'ar \cite{ML2} concerning the Wasserstein mean. We prove the conjecture for $I_2$-type algebras in regard of the Wasserstein mean, too. We also give two  conditions that characterise centrality in $C^*$-algebras. 
\end{abstract}

\section{Introduction}
A preserver is a map $\Phi$ from a set $\mathcal{X}$ to a set $\mathcal{Y}$ (often $\mathcal{Y}=\mathcal{X})$ such that $\Phi$ preserves a subset of elements, operation, quantity or structure of $\mathcal{X}$. Preservers arise naturally in most areas of mathematics, for example homomorphisms, isomorphisms in algebra, isometries in the study of metric spaces, homeomorphisms, diffeomorphisms in topology etc. However, the area where preserver problems are systematically studied is mostly matrix theory, and its infinite dimensional counterpart, operator theory. Linear preserver problems (where $\Phi$ is assumed to be linear) have  distinguished importance among preserver problems. For example, one of the most well-known results of this type is due to Frobenius who showed in \cite{Frob} that if $\Phi: M_n(\C) \to M_n(\C)$ is linear and satisfies $\mathrm{det}(A)=\mathrm{det}(\Phi(A))$ for all $A \in M_n(\C)$ then $\Phi$ is necessarily of the form
$$\Phi(A)=MAN \,\, (A \in M_n(\C))$$
or 
$$\Phi(A)=MA^{\mathrm{tr}}N \,\, (A \in M_n(\C))$$
where $A^{\mathrm{tr}}$ denotes the transpose of $A$ and $M,N$ are non-singular matrices such that $\mathrm{det}(MN)=1$. A more recent result of similar nature is due to Dolinar and \v{S}emrl. They showed in \cite{DS} that if $\Phi:M_n(\C) \to M_n(\C)$ is a surjective map satisfying
$\mathrm{det}(A+\lambda B)=\mathrm{det}(\Phi(A)+\lambda \Phi(B))$ for all $A,B \in M_n(\C), \, \lambda \in \C$, then $\Phi$ is necessarily of the form of
$$\Phi(A)=MAN \,\, (A \in M_n(\C))$$
or of the form of
$$\Phi(A)=MA^{\mathrm{tr}}N \,\, (A \in M_n(\C))$$
where $M,N$ are nonsingular matrices such that $\mathrm{det}(MN)=1$.
Another celebrated result about linear preservers is due to Banach and Stone. Their theorem concerns surjective linear isometries between continuous function algebras of compact Hausdorff spaces \cite{Banach}, \cite{Stone}.
\begin{theoremlett} \label{thm:banach-stone}
Let $X,Y$ be compact Hausdorff spaces and let $C(X),C(Y)$ denote the continuous function algebras of $X$ and $Y$ with respect to the supremum norm. Let $T:C(X) \to C(Y)$ be a surjective linear isomorphism. Then there exists a homeomorphism $\varphi:Y \to X$ and a function $g \in C(Y)$ with $|g(y)|=1 \,\, \forall y \in Y$ such that $(Tf)(y)=g(y)f(\varphi(y)) \,\, \forall f \in C(X), y \in Y$. 
\end{theoremlett}
A Jordan-homomorphism is a linear map $J$ between algebras $\mathcal{A}$ and $\mathcal{B}$ such that $J(ab+ba)=J(a)J(b)+J(b)J(a), \,\, \forall a,b \in \mathcal{A}$. Elementary calculations show that a linear map $J$ is a Jordan-homomorphism if and only if $J(a^2)=J(a)^2 \,\, \forall a \in \mathcal{A}$, and all Jordan-homomorphisms preserve the Jordan triple product, that is, $J(bab)=J(b)J(a)J(b) \,\, \forall a,b \in \mathcal{A}$.
We give a theorem of Kadison from \cite{Kadison} as an example of a theorem concerning Jordan-homomorphisms which preserve a subset of elements of a $C^*$-algebra.
\begin{theoremlett} \label{thm:kadison}
Let $\mathcal{A}$ and $\mathcal{B}$ be unital $C^*$-algebras, and let $J: \cA \to \mathcal{B}$ be a bijective Jordan-homomorphism that maps precisely the self-adjoint elements of $\cA$ to self-adjoint elements of $\mathcal{B}$. Then $J$ is an isometry that preserves commutativity, that is if $ab=ba, \,\, a,b \in \cA$ then $J(a)J(b)=J(b)J(a)$. 
\end{theoremlett}
We now turn to the preservers of Kubo-Ando and other types of means, which are the main content of this paper.
Let $H$ be a complex Hilbert space, and denote by $B(H)$  the $C^*$-algebra of all bounded operators on $H$ and by $B(H)^+$ the convex cone of positive semidefinite operators with respect to the Löwner order. A Kubo-Ando mean is a binary operation $\sigma$ on $B(H)^+$ which satisfies the following conditions:
\begin{enumerate}
    \item $I\sigma I=I$
    \item if $ A \leq B$ and $C \leq D$ then $A \sigma B \leq C \sigma D$
    \item $C (A \sigma B)C \leq (CAC) \sigma (CBC)$
    \item $A_n \downarrow A$ and $B_n \downarrow B$ strongly, then $A_n \sigma B_n \downarrow A \sigma B$, where $A_n \downarrow A$ denotes monotone decreasing convergence in the Löwner order.
\end{enumerate}
A theory of such means was introduced by Kubo and Ando in \cite{KA}. They showed in Theorem 3.2 that for an infinite dimensional $H$ there is a one-to-one correspondence between Kubo-Ando means on $B(H)^+$ and the collection of operator monotone functions $f:(0,+\infty) \to (0,+\infty)$ with $f(1)=1$. The operator monotone function that corresponds to the Kubo-Ando mean $\sigma$ is given by the formula $f(t)I=I \sigma tI, \, t>0$. Conversely, the Kubo-Ando mean $\sigma$ corresponding to the operator monotone function $f$ is given by the formula $A \sigma B=A^{\fel}f(A^{-\fel}BA^{-\fel})A^{\fel}$. Because of the continuous function calculus, this formula makes sense in the setting of general $C^*$-algebras, too. The most distinguished Kubo-Ando means are the harmonic, geometric and arithmetic means. The corresponding operator monotone functions are $t \to \frac{2t}{1+t},\, t \to \sqrt{t},\, t \to \frac{1+t}{2} \, (t>0)$, respectively. Using the formula $A \sigma B=A^{\fel}f(A^{-\fel}BA^{-\fel})A^{\fel}$ these means turn out to be the following for invertible $A,B \in B(H)^+$:
$$A!B=2(A^{-1}+B^{-1})^{-1}, \, A\#B=A^{\fel}(A^{-\fel}BA^{-\fel})^{\fel}A^{\fel},\, A \nabla B=\frac{A+B}{2}.$$
The family of Kubo-Ando means we are mostly concerned about in this paper is the family of power means. They are defined for $p \in [-1,1]\setminus{\{0\}}$ as 
\be \label{eq:KA-def}
A \emp B=A^{\fel}\left(\frac{I+(A^{-\fel}BA^{-\fel})^p}{2}\right)^{\frac{1}{p}}A^{\fel}.
\ee
the corresponding operator monotone functions are of the form of $t \to (\frac{1+t^p}{2})^{\frac{1}{p}}$.
Note that both the Kubo-Ando power means and Kubo-Ando geometric mean have interesting non-Kubo-Ando counterparts. 
The so called conventional $p$-power means \cite{ML1} are obtained by applying function calculus naively and are defined as
$$
A m_p B=\ler{\frac{A^p+B^p}{2}}^{\frac{1}{p}}.
$$
A type of geometric mean(called spectral geometric mean) was introduced by Fiedler and Pták in \cite{FPT}.
It is defined as 
$$A \sharp B=(A^{-1} \#B)^{\fel}A(A^{-1} \#B)^{\fel}.$$
Another mean, called the Wasserstein mean was described in \cite{Bhat1}. It is defined as
$$A \sigma_W B=\frac{A+B+A \cdot A^{-1}\#B+A^{-1}\#B \cdot A}{4}.$$
It coincides with the conventional $p$-power mean for $p=1/2$ for commuting $A$ and $B$.
Note that the Kubo-Ando geometric mean and the Wasserstein mean have strong connections with Riemannian and Finsler geometry.
If the $C^*$-algebra in consideration is $M_n(\C)$, then there is a natural Riemannian metric on the positive definite cone of $M_n(\C)$.
For positive definite $A,B$ the Kubo-Ando geometric mean $A\#B$ is the midpoint of the unique geodesic curve 
$$t \to A^{\fel}(A^{-\fel}BA^{-\fel})^tA^{\fel}$$
connecting $A$ and $B$ in this Riemannian structure. In the more general setting of $C^*$-algebras, a Finsler type structure can be given to the positive definite cone. For more details, see the papers \cite{CPR1}, \cite{CPR2}, \cite{CPR3}, \cite{CPR4} and Chapter 6. of \cite{Bhatia}.
Recall that there is an important metric on the positive definite cone of $M_n(\C)$, called the Bures-Wasserstein metric.
It is widely used in quantum information theory and in the theory of optimal transport and it is defined as
$$d_{BW}(A,B)=\ler{\Tr(A)+\Tr(B)-2\Tr(A^{\fel}BA^{\fel})^{\fel}}^{\frac{1}{2}}$$
for positive definite matrices $A,B$. Here $\Tr$ stands for the standard trace functional on $M_n(\C)$. 
It was observed in \cite{Bhat1} that there is a Riemannian geometry on the positive definite cone of $M_n(\C)$, whose geodesic distance is exactly the Bures-Wasserstein metric, and the curve connecting two matrices $A$ and $B$ is 
$$t \to (1-t)^2A+t^2B+t(1-t)(A(A^{-1}\#B)+(A^{-1}\#B)A),\,\, t \in [0,1].$$
The midpoint of this curve is exactly the Wasserstein mean of $A$ and $B$.
The means mentioned above have been actively researched in the past decades, including from the viewpoint of preserver problems.
We give a very brief overview of results from this area.
A description of geometric mean preserving maps on Hilbert spaces was given in \cite{ML5}. If $H$ is a complex Hilbert space with $dim\,H \geq 2$, and $\Phi: B(H)^+ \to B(H)^+$ is a bijective map that preserves the geometric mean, ie. $\Phi(A\#B)=\Phi(A)\#\Phi(B), \,\, A,B \in B(H)^+ $, then $\Phi$ is necessarily of the form of 
$$\Phi(A)=SAS^* \,\, A \in B(H)^+,$$
for some  bounded, invertible linear or conjugate linear operator $S$ on $H$. The symbol $A\#B$ stands here for the variational expression of the geometric mean,
$$A \#B=\max \left\{X \geq 0: \begin{bmatrix} A & X \\ X & B \end{bmatrix} \geq 0 \right\}$$
introduced by Ando in \cite{Ando}. It coincides with the Kubo-Ando geometric mean for invertible operators.
In \cite{ML4}, Moln\'ar showed that if $\Phi$ is a bijective map on the positive definite cone of $B(H)$ for some complex Hilbert space $H$ that preserves the harmonic mean, ie. $\Phi(A!B)=\Phi(A)!\Phi(B)$ for all $A,B \in B(H)^{++}$, then there exists a bounded linear or cojugate linear operator $S$ on $H$ such that $$\Phi(A)=SAS^*, \,\, A \in B(H)^{++}.$$
Preservers of the spectral geometric mean on Hilbert spaces were described in \cite{LMW}. If $H$ is a complex Hilbert space, and $\Phi: B(H)^{++} \to B(H)^{++}$ is a continuous bijective map satisfying 
$$\Phi(A \sharp B)=\Phi(A) \sharp \Phi(B), \,\, A,B \in B(H)^{++}$$
and in addition, $\Phi(I)=I$ and $\Phi$ has a continuous bijective extension to $B(H)^+$ then there exists a unitary or antiunitary operator $U$ on $H$ such that
$$\Phi(A)=UAU^*, \,\, A \in B(H)^{++}.$$
Let us now turn to the results that motivated the present work. Functionals(scalar valued maps) have a distinguished role in functional analysis.
In \cite{ML1}, Moln\'ar showed that if $\cA$ is a von Neumann algebra without $I_1,I_2$ type direct summands, then there are only trivial functionals on the positive definite cone of $\cA$ that preserve the Kubo-Ando power mean for $p \in (-1,1)$, ie. if $f:\cA^{++} \to (0, +\infty)$, such that $f(A \emp B)=f(A) \emp f(B), \, A,B \in \cA^{++}$, then $f \equiv c$ for some positive constant $c$. In commutative algebras, the conventional and Kubo-Ando power means coincide. The conventional power means have non-constant preserving functionals, their complete description was given in Proposition 5. in \cite{ML1}.
The absence of $I_2$ type algebras in the statement comes from the nature of the proof: it heavily relies on the solution of the Mackey-Gleason problem due to Brunce and Wright. It was conjectured in \cite{ML1}, that the result remains true for $I_2$ type algebras, too.
A result and a conjecture of similar nature were formulated in \cite{ML2} concerning the Wasserstein mean. We confirm these two conjectures about the preserver functionals of the Wasserstein and Kubo-Ando power means on $M_2(\C)$.

\section{Basic notions, notation}
We denote abstract $C^*$-algebras with $\cA$. All $C^*$-algebras are assumed to have an identity element, denoted by $I$.  We denote by $\cA_{sa}, \cA^+, \cA^{++}$, the real vector space of self-adjoint elements, the positive semidefinite, and the positive definite cone of the algebra $\cA$.
We use the same notation in regard of the particular $C^*$-algebra $M_2(\mathbb{C})$.

\section{Main result}

\begin{theorem} \label{thm:main}
Set $p\in (-1,1)\setminus{\{0\}}$ and let $f: M_2^{++}\ler{\C} \rightarrow (0, \infty)$ satisfy
\be \label{eq:pres-eq}
f\ler{A \emp B} = f(A) \emp f(B) \qquad \ler{A, B \in  M_2^{++}\ler{\C}}
\ee
where $\emp$ denotes the Kubo-Ando $p$-power mean. Then $f \equiv \mathrm{const}>0.$
\end{theorem}

\begin{proof}
We first note that if $f$ is a solution of \eqref{eq:pres-eq}, then so is $\lambda f$ for any $\lambda>0.$ Therefore, without loss of generality assume that $f(I)=1.$ Let us define
\be \label{eq:phidef}
\varphi=\ler{.}^p \circ f \circ \ler{.}^{1/p}.
\ee
Then the preserver equation \eqref{eq:pres-eq} reads in terms of $\varphi$ as
\be \label{eq:pres-eq-phi}
\varphi\ler{\ler{A \emp B}^p}=\frac{\varphi\ler{A^p}+\varphi\ler{B^p}}{2} \qquad \ler{A, B \in  M_2^{++}\ler{\C}}.
\ee
In particular, if $A_1$ and $A_2$ commute, then
\be \label{eq:phi-a1-a2}
\varphi\ler{\frac{A_1^p+A_2^p}{2}}=\frac{\varphi\ler{A_1^p}+\varphi\ler{A_2^p}}{2}.
\ee
A maximal Abelian subalgebra of $M_2\ler{\C}$ is clearly of the form $\cA=\mathrm{linspan}\langle P, I-P\rangle$ for some rank-one ortho-projection $P,$ and hence it is of the form $\cA=\mathrm{linspan}\langle I, G \rangle$ for some traceless self-adjoint unitary $G,$ where $G$ is given by $G=2P-I.$ We denote the intersection of $\cA$ and $M_2^{++}\ler{\C}$ by $\cA^{++}.$ Note that the $p$-th power function is a bijection on every $\cA^{++}.$ Therefore, \eqref{eq:phi-a1-a2} tells us that regardless of the choice of the MASA $\cA$, the map $\varphi: \cA^{++} \rightarrow (0, \infty)$ is a Jensen map, that is, it preserves the arithmetic mean.
\par
We show that the Jensen property, together with positivity implies continuity. We mostly follow the argument of Proposition 5. in \cite{ML1}.
Since $\cA^{++}$ is a $\mathbb{Q}$-convex subset of the $\mathbb{Q}$-linear space $\cA$, and $(0, +\infty)$ is a $\mathbb{Q}$-convex subset of the $\mathbb{Q}$-linear space $\mathbb{R}$, we obtain from \cite{Ger}, that $\varphi$ is necessarily of the form $\varphi(x)=L(x)+c$, for some fixed $c \in (0, +\infty)$ and additive function $L: \cA^{++} \to \mathbb{R}$. Since $\varphi$ takes positive values, it follows that $L$ is non-negative. Indeed, $\varphi$ is bounded from below on $\cA^{++}$, so we get that $nL(A)=L(nA) \geq m$, for some $m \in \mathbb{R}$. After dividing by $n$, and letting it tend to infinity we get that $L(A) \geq 0$. For a fixed $A \in \cA^{++}$ the additive function $t \to L(tA)$ on $\mathbb{R}$ is nonnegative on the positive half line. Theorem 9.3.1 in \cite{Kuczma} states that any additive function of the reals which is bounded from below on a set of positive Lebesgue measure is necessarily continuous. We claim that $L$ is continuous as an  $\cA^{++} \rightarrow \mathbb{R}$ function. Indeed, let $P$ and $Q$ be two perpendicular orthogonal projections, generating the commutative subalgebra $\cA$. Then $X=I+P$ and $Y=I+Q$ are linearly independent positive definite elements, thus any $A \in \cA^{++}$ can be expressed uniquely in the form of $A=uX+vY$.
If $A_n=u_nX+v_nY$ is a sequence converging to $A$, we must have that $u_n \to u$ and $v_n \to v$. We have that $|L(A_n)-L(A)|=|L(u_nX+v_nY)-L(uX+vY)|=|L(u_nX)-L(uX)+L(v_nY)-L(vY)| \leq |L(u_nX)-L(uX)|+|L(v_nY)-L(vY)|$. Since $t \to L(tX)$ is continuous, $L(u_nX)$ tends to $L(uX)$ as $u_n$ tends to $u$, and the same is true for $v_n,\,v$ and $Y$. Thus we have that $L$ is a continuous function from $\cA^{++}$.
\par
So $\varphi$ is affine and non-negative on every commutative sub-cone. Consequently, using the basis $\left\{I,G\right\}$ of $\cA$ we can expand $\varphi$ the following way: there exists a non-negative constant $c_I,$ and for every traceless self-adjoint unitary $G$ there exists a constant $c_G$ with $\abs{c_G} \leq c_I$ (this inequality follows from the non-negativity of $\varphi$) such that
\be \label{eq:phi-comm}
\varphi\ler{t I + s G}=c_I t + c_G s +(1-c_I) \qquad \ler{0\leq \abs{s}<t}.
\ee
Note that $c_G \equiv c_{-G}.$ Now let us consider the commutative sub-algebras generated by the specific traceless self-adjoint unitaries
$$
\sigma_z=\left[ \ba{cc} 1 & 0 \\ 0 & -1 \ea\right] \text{ and } \sigma_x=\left[ \ba{cc} 0 & 1 \\ 1 & 0 \ea\right].
$$
It will be useful in the sequel that conjugation with the unitary matrix
\be\label{eq:U-def}
U=\frac{1}{\sqrt{2}}\ler{\sigma_z+\sigma_x}=\frac{1}{\sqrt{2}}\left[ \ba{cc} 1 & 1 \\ 1 & -1 \ea\right]
\ee 
maps $\sigma_x$ and $\sigma_z$ to each other, that is,
\be \label{eq:U-switch}
U \sigma_z U^*=\sigma_x \text{ and } U \sigma_z U^*=\sigma_x.
\ee
Now let
\be \label{eq:abdef}
A=t I + s \sigma_z \text{ and } B=t I + s \sigma_x
\ee
for some $0\leq \abs{s}<t.$ 
By this definition and \eqref{eq:U-switch} we have
$$
U A U^*=B \text{ and } U B U^*=A.
$$
Consequently, using that functional calculus is compatible with unitary conjugations,
$$
U \ler{A \emp B} U^*=\ler{U A U^*}^{\fel}\left(\frac{I+U (A^{-\fel}BA^{-\fel})^p U^*}{2}\right)^{\frac{1}{p}}\ler{U A U^*}^{\fel}
$$
\be \label{eq:unit-eq}
=B^{\fel}\left(\frac{I+ (\ler{U A U^*}^{-\fel}\ler{U B U^*}\ler{U A U^*}^{-\fel})^p}{2}\right)^{\frac{1}{p}}B^{\fel}=B \emp A=A \emp B.
\ee
Equivalently,
$$
U A \emp B = A \emp B U,
$$
which means that $A \emp B$ is in the commutative sub-algebra generated by $U=\frac{1}{\sqrt{2}}\ler{\sigma_z+\sigma_x}$ for every possible choices of $t$ and $s$ in \eqref{eq:abdef}.
By this feature of the $p$-power mean for the family of matrix pairs
\be \label{eq:abepsdef}
A_\eps:= I+\eps \sigma_z \text{ and } B_\eps:= I+\eps \sigma_x \qquad \ler{\eps \in (-1,1)}
\ee
and by \eqref{eq:phi-comm}, we have
\be \label{eq:phi-aff}
\varphi\ler{\ler{A_\eps \emp B_\eps}^p}= c_I \fel \tr \ler{A_\eps \emp B_\eps}^p+c_U \fel\tr \ler{U\ler{A_\eps \emp B_\eps}^p}+\ler{1-c_I}.
\ee
Let us introduce the notation
\be \label{eq:g-p-def}
g_p(x)=\ler{\frac{1+x^p}{2}}^{1/p}.
\ee
We compute the left hand side of \eqref{eq:phi-aff} in the small $\eps$ regime up to second order. We start with estimating $A_\eps^{-\fel}$ with its binomial series expansion around the identity:
$$
A_{\eps}^{-\fel} B_{\eps} A_{\eps}^{-\fel}=
$$
$$
=\ler{I-\fel \eps \sigma_z+\frac{3}{8}\eps^2I+\kif}\ler{I+\eps\sigma_x}\ler{I-\fel \eps \sigma_z+\frac{3}{8}\eps^2I+\kif}
$$
$$
=I-\fel \eps \sigma_z-\fel \eps \sigma_z+\frac{1}{4}\eps^2I+\frac{3}{8}\eps^2 I+\frac{3}{8}\eps^2 I+\eps\sigma_x-\fel\eps^2\sigma_z\sigma_x-\fel\eps^2\sigma_x\sigma_z+\kif.
$$
Since $\sigma_z\sigma_x=-\sigma_x\sigma_z$, those terms cancel out, and after arranging the other terms together we arrive at
$$
A_{\eps}^{-\fel} B_{\eps} A_{\eps}^{-\fel}=I+\eps(\sigma_x-\sigma_z)+\eps^2 I+\kif.
$$
Computing the first and second order derivatives of $g_p(x)=\ler{\frac{1+x^p}{2}}^{\frac{1}{p}}$ gives that
$g_p'(x)=\frac{1}{p} \cdot \ler{\frac{1+x^p}{2}}^{\frac{1}{p}-1} \cdot \frac{px^{p-1}}{2}=\fel \cdot \ler{\frac{1+x^p}{2}}^{\frac{1}{p}-1} \cdot x^{p-1},$ hence $g_p'(1)=\fel,$ and
$g_p''(x)=\fel\ler{\frac{1}{p}-1}\ler{\frac{1+x^p}{2}}^{\frac{1}{p}-2}\cdot \frac{x^{p-1}}{2} \cdot x^{p-1}+\fel\ler{\frac{1+x^p}{2}}^{\frac{1}{p}-1} \cdot (p-1)x^{p-2}$, which leads us to $g_p''(1)=\frac{1}{4}\ler{\frac{1}{p}-1}+\frac{p-1}{2}.$
Therefore,
$$
g_p\ler{A_{\eps}^{-\fel} B_{\eps} A_{\eps}^{-\fel}}=I+\fel\eps(\sigma_x-\sigma_z)+\fel\eps^2I+\fel\ler{\frac{1}{4}\ler{\frac{1}{p}-1}+\frac{p-1}{2}}\eps^2 2I+\kif
$$
$$
=I+\fel\eps(\sigma_x-\sigma_z)+\fel\eps^2\ler{1+\fel\ler{\frac{1}{p}-1}+p-1}+\kif
$$
$$
=I+\fel\eps(\sigma_x-\sigma_z)+\fel\eps^2\ler{\fel\ler{\frac{1}{p}-1}+p}I+\kif.
$$
Since $A_{\eps}\emp B_{\eps}=A_{\eps}^{\fel}g_p\ler{A_{\eps}^{-\fel} B_{\eps} A_{\eps}^{-\fel}}A_{\eps}^{\fel}$, it follows that 
$$
A_{\eps}\emp B_{\eps}=\ler{I+\fel \eps \sigma_z-\frac{1}{8}\eps^2 I+\kif}
\times
$$
$$
\times \ler{I+\fel\eps(\sigma_x-\sigma_z)+\fel\eps^2\ler{\fel\ler{\frac{1}{p}-1}+p}I+\kif}\times
$$
$$
\times \ler{I+\fel \eps \sigma_z-\frac{1}{8} \eps^2 I+\kif}=
$$
$$
=I+\fel\eps\sigma_z-\frac{1}{8}\eps^2I+\fel\eps\sigma_z+\frac{1}{4}\eps^2I-\frac{1}{8}\eps^2I+\fel\eps(\sigma_x-\sigma_z)+
$$
$$
+\frac{1}{4}\eps^2\sigma_z(\sigma_x-\sigma_z)+\frac{1}{4}\eps^2(\sigma_x-\sigma_z)\sigma_z+\fel\eps^2\ler{p+\fel\ler{\frac{1}{p}-1}}I+\kif
$$
$$
=I+\fel\eps(\sigma_z+\sigma_x)+\eps^2\ler{-\frac{1}{8}+\frac{1}{4}-\frac{1}{8}-\frac{1}{4}-\frac{1}{4}+\frac{p}{2}+\frac{1}{4p}-\frac{1}{4}}I+\kif
$$
$$
=I+\fel\eps(\sigma_z+\sigma_x)+\ler{\frac{p}{2}+\frac{1}{4p}-\frac{3}{4}}\eps^2I+\kif.
$$
From this, we get that
$$
\ler{A_{\eps} \emp B_{\eps}}^p=I+\frac{p}{2}\eps(\sigma_z+\sigma_x)+\ler{\frac{p^2}{2}+\frac{1}{p}-\frac{3p}{4}}\eps^2I+\frac{p(p-1)}{2}\frac{\eps^2}{4}2I+\kif.
$$
Therefore, by \eqref{eq:phi-aff} we get that
$$
\varphi\ler{\ler{A_{\eps} \emp B_{\eps}}^p}=
$$
\be \label{eq:pres-lhs}
=c_I\ler{1+\ler{\frac{p^2}{2}+\frac{1}{4}-\frac{3p}{4}+\frac{p(p-1)}{4}}\eps^2}+c_U \frac{p\eps}{\sqrt{2}}+\kif+(1-c_I).
\ee
On the other hand, the binomial expansion of $A_\eps$ and $B_\eps$ around the identity gives 
$$
A_{\eps}^p=(I+\eps\sigma_z)^p=I+p\eps\sigma_z+\frac{p(p-1)}{2}\eps^2I+\kif
$$
and
$$
B_{\eps}^p=(I+\eps\sigma_x)^p=I+p\eps\sigma_x+\frac{p(p-1)}{2}\eps^2I+\kif.
$$
By \eqref{eq:phi-comm} this implies
\be \label{eq:pres-rhs-1}
\varphi(A_{\eps}^p)=c_I \ler{1+\frac{p(p-1)}{2}\eps^2}+c_{\sigma_z}p\eps+\kif+(1-c_I)
\ee
and
\be \label{eq:pres-rhs-2}
\varphi(B_{\eps}^p)=c_I \ler{1+\frac{p(p-1)}{2}\eps^2}+c_{\sigma_x}p\eps+\kif+(1-c_I)
\ee
We substitute $A_\eps$ and $B_\eps$ to the preserver equation \eqref{eq:pres-eq-phi} and get 
\be \label{eq:(PE)}
\varphi\ler{\ler{A_\eps \emp B_\eps}^p}=\frac{\varphi(A_\eps^p)+\varphi(B_\eps^p)}{2}.
\ee
By \eqref{eq:pres-lhs},\eqref{eq:pres-rhs-1}, and \eqref{eq:pres-rhs-2} we get that \eqref{eq:(PE)}  is equivalent to 
$$
c_I\ler{1+\ler{\frac{p^2}{2}+\frac{1}{4}-\frac{3p}{4}+\frac{p(p-1)}{4}}\eps^2}+c_U \cdot \frac{p\eps}{\sqrt{2}}+\kif=
$$
\be \label{eq:pres-fin}
=c_I \ler{1+\frac{p(p-1)}{2}\eps^2}+\frac{p\eps}{\sqrt{2}}\frac{c_{\sigma_z}+c_{\sigma_x}}{\sqrt{2}}+\kif.
\ee
Comparing the first order terms in \eqref{eq:pres-fin} gives that
$$
c_U=c_{\frac{1}{\sqrt{2}}(\sigma_x+\sigma_z)}=\frac{1}{\sqrt{2}}(c_{\sigma_x}+\sigma_z),
$$
while comparing the second order terms gives
\be \label{eq:sec-ord}
c_I \ler{\frac{p^2}{2}+\frac{1}{4}-\frac{3p}{4}+\frac{p(p-1)}{4}-\frac{p(p-1)}{2}}=0.
\ee
Simple calculations show that \eqref{eq:sec-ord} is equivalent to $c_I \cdot (p-1)^2=0$.
This means that if $p \neq 1$, than $c_I=0$, which means that $c_G=0$ for all self-adjoint traceless unitary $G$, which implies that $f \equiv \text{constant},$ as desired.
If $p=1$, then all positive linear functionals satisfy \eqref{eq:(PE)}.
\end{proof}
We summarise our knowledge about Kubo-Ando power mean preserving functionals in the following corollary.
\begin{corollary}
Let $\cA$ be a von Neumann algebra without type $I_1$ direct summands. If $p \in (-1,1), p\neq 0$, then any function $f: \cA^{++} \to (0,+\infty)$ satisfying $f(A \emp B)=f(A) \emp f(B)$ is necessarily constant.
\begin{proof}
The case omitting $I_2$-type algebras has been already  dealt with in Theorem 8. in \cite{ML1}, and we have just proven the statement for $I_2$-type algebras.  
\end{proof}
\end{corollary}
We follow a similar program concerning preserver functions of the Wasserstein-mean.
\begin{theorem} \label{thm:main-wass}
Let $f: M_2^{++}\ler{\C} \rightarrow (0, \infty)$ satisfy
\be\label{pres:wass}
f\ler{A \sigma_W B}=f(A) \sigma_W f(B) \qquad \ler{A, B \in  M_2^{++}\ler{\C}}
\ee
where $\sigma_W$ denotes the Wasserstein-mean. Then $f$ is necessarily constant.
\end{theorem}
\begin{proof}
Note that if $f$ satisfies \eqref{pres:wass}, then so does $\lambda f$ for any $\lambda>0$. Therefore, without the loss of generality we can assume that $f(I)=1$.
Now let $\varphi$ be defined as $\varphi=(.)^{\fel} \circ f \circ (.)^2$.
Then the preserver equation \eqref{pres:wass} in terms of $\varphi$ reads as
\be\label{pres:varphi}
\varphi(\ler{A \sigma_W B}^{\fel})=\frac{\varphi(A^{\fel})+\varphi(B^{\fel})}{2} \qquad \ler{A, B \in  M_2^{++}\ler{\C}}
\ee
As we mentioned before, a straightforward calculation shows that if $A$ and $B$ commute, their Wasserstein mean coincides with the conventional power mean for $p=1/2$.
Therefore, for commuting $A,B$ \eqref{pres:varphi} reads as
$$\varphi\ler{\frac{A^{\fel}+B^{\fel}}{2}}=\frac{\varphi(A^{\fel})+\varphi(B^{\fel})}{2} \qquad \ler{A, B \in  M_2^{++}\ler{\C}}.$$
Since the square root function is a bijection on the positive definite cone, this can be rewritten as 
\be\label{affin}
\varphi \ler{\frac{A+B}{2}}=\frac{\varphi(A)+\varphi(B)}{2} \qquad \ler{A, B \in  M_2^{++}\ler{\C}}.
\ee
We get that $\varphi$ preserves the arithmetic mean of commuting elements. Note that we are in the same postion as we were  in the proof of Theorem \ref{thm:main}. We can repeat the reasoning from there, and we obtain that $\varphi$ must be continuous on all maximal Abelian subalgebras.
Therefore, $\varphi$ is affine on every commutative sub-cone, so there exists a non-negative constant $c_I$ and for every traceless self-adjoint unitary $G$ there exists a constant $c_G$ such that $|c_G| \leq c_I$ such that 
\be\label{eq:aff}
\varphi(tI+sG)=c_It+c_Gs+(1-c_I) \qquad (0 \leq |s|<t).
\ee
Now let, $\sigma_z,\sigma_x, U$ be as in Theorem \ref{thm:main}. Using that functional calculus is compatible with unitary conjugations we get that
$$U A \sigma_WB U^*=\ler{UAU^*}^{-\fel} \ler{\frac{UAU^*+U(A^{\fel}BA^{\fel})^{\fel}U^*}{2}}^2\ler{UAU^*}^{-\fel}$$
$$=B^{-\fel}\ler{\frac{B+(B^{\fel}AB^{\fel})^{\fel}}{2}}^2B^{-\fel}=B \sigma_W A=A \sigma_W B.$$
This means that $U A \sigma_W B=A \sigma_W B U$.In other words, $A \sigma_W B$ is in the commutative sub-algebra generated by $U=\frac{1}{\sqrt{2}}(\sigma_x+\sigma_z)$ for every possible choices of $s$ and $t$.
Now, by \eqref{affin} we get that
$$\varphi\ler{(A_{\eps}\sigma_W B_{\eps})^{\fel}}=c_I \fel \tr(A_{\eps} \sigma_W B_{\eps})^{\fel}+c_U \fel \tr(U(A_{\eps} \sigma_W B_{\eps})^{\fel})+(1-c_I).$$
Following the notation of Theorem \ref{thm:main}, let $A_{\eps}=I+\eps \sigma_z$, and $B_{\eps}=I+\eps \sigma_x$. Then
$A_{\eps}^{\fel}=I+\fel\eps\sigma_z-\frac{1}{8}\eps^2I+\kif$ and 
$$
A_{\eps}^{\fel}B_{\eps}A_{\eps}^{\fel}=
$$
$$
=I+\fel \eps \sigma_z-\frac{1}{8}\eps^2I+\eps \sigma_x+\fel\eps^2 \sigma_x\sigma_z+\fel \eps \sigma_z+\frac{1}{4}\eps^2\sigma_z^2+\fel\eps^2\sigma_z\sigma_x-\frac{1}{8}\eps^2I+\kif=
$$
$$
=I+\eps(\sigma_z+\sigma_x)+\kif.
$$
Then 
$$\ler{A^{\fel}BA^{\fel}}^{\fel}=I+\fel \eps(\sigma_z+\sigma_x)-\frac{1}{8}\eps^2(\sigma_z+\sigma_x)^2+\kif$$
$$=I+\frac{1}{2}\eps(\sigma_z+\sigma_x)-\frac{1}{4}\eps^2I+\kif.$$
Note that the Wasserstein mean of a pair $A,B$ can be expressed in the following form:
\be\label{eq:was-altern}
A \sigma_W B=\frac{A+B+A^{-\fel}\ler{A^{\fel}BA^{\fel}}^{\fel}A^{\fel}+A^{\fel}\ler{A^{\fel}BA^{\fel}}^{\fel}A^{-\fel}}{4}.
\ee
We compute the different components of the expression above separately and then add them together.
$$
A_{\eps}^{-\fel}\ler{A_{\eps}^{\fel}B_{\eps}A_{\eps}^{\fel}}^{\fel}A_{\eps}^{\fel}=
$$
$$
=\ler{I-\fel \eps \sigma_z+\frac{3}{8}\eps^2I+\kif}\ler{I+\frac{\eps}{2}(\sigma_z+\sigma_x)-\frac{1}{4}\eps^2I+\kif} \times
$$
\be \label{eq:was1}
\times\ler{I+\fel\eps\sigma_z-\frac{1}{8}\eps^2I+\kif}=I+\frac{\eps}{2}(\sigma_z+\sigma_x)+\frac{\eps^2}{2} \cdot \sigma_x\sigma_z+\kif.
\ee

Note that the other term, $A_{\eps}^{\fel}\ler{A_{\eps}^{\fel}B_{\eps}A_{\eps}^{\fel}}^{\fel}A_{\eps}^{-\fel}$ is the adjoint of what we just computed, therefore it equals to $I+\frac{\eps}{2}(\sigma_z+\sigma_x)+\frac{\eps^2}{2} \cdot \sigma_z\sigma_x+\kif$.
From these, we get that the Wasserstein mean of $A_{\eps},\,B_{\eps}$ is
$$
A_{\eps} \sigma_W B_{\eps}=\frac{I+\eps\sigma_z+I+\eps\sigma_x+I+\frac{\eps}{2}(\sigma_z+\sigma_x)+I+\frac{\eps}{2}(\sigma_z+\sigma_x)+\kif}{4}=
$$
\be\label{eq:wassemifinal}
=I+\frac{1}{2}\eps(\sigma_z+\sigma_x)+\kif.
\ee
Now we get that 
$$
\ler{A_{\eps} \sigma_W B_{\eps}}^{\fel}=I+\frac{1}{4}\eps(\sigma_z+\sigma_x)-\frac{1}{32}\eps^2(\sigma_z+\sigma_x)^2+\kif=
$$
\be \label{eq:wasfinal}
=I+\frac{1}{4}\eps(\sigma_z+\sigma_x)-\frac{1}{16}\eps^2I+\kif.
\ee
We get that 
$$\varphi((A_{\eps} \sigma_W B_{\eps})^{\fel})=c_I \cdot (I-\frac{1}{16}\eps^2)+c_U \frac{\sqrt{2}}{4}\eps+\kif+(1-c_I).$$
The binomial expansion of $A_{\eps}^{\fel}$ and $B_{\eps}^{\fel}$ around the identity gives us 
$$A_{\eps}^{\fel}=I+\fel\eps\sigma_z-\frac{1}{8}\eps^2I+\kif$$
$$B_{\eps}^{\fel}=I+\fel\eps\sigma_x-\frac{1}{8}\eps^2I+\kif.$$
From \eqref{eq:aff} it follows that 
$$\varphi(A_{\eps}^{\fel})=c_I(1-\frac{1}{8}\eps^2)+\fel\eps c_{\sigma_z}+(1-c_I)+\kif$$
$$\varphi(B_{\eps}^{\fel})=c_I(1-\frac{1}{8}\eps^2)+\fel\eps c_{\sigma_x}+(1-c_I)+\kif.$$
From \eqref{pres:varphi} we obtain that
$$c_I \cdot (1-\frac{1}{16}\eps^2)+c_U \frac{\sqrt{2}}{4}\eps=c_I(1-\frac{1}{8}\eps^2)+\frac{1}{4}\eps(c_{\sigma_z}+c_{\sigma_x}).$$
This implies that $c_I(\frac{1}{8}\eps^2-\frac{1}{16}\eps^2)=0$, which means that $c_I=0$, therefore $c_G \equiv 0$ for all self-adjoint traceless unitary $G$, so $f \equiv \text{constant}$ as we claimed.
\end{proof}
We can formulate a corollary similar to the case of Kubo-Ando mean preserving functionals.
\begin{corollary}
Let $\cA$ be a von Neumann algebra without type $I_1$ direct summands. Then any function $f: \cA^{++} \to (0,+\infty)$ satisfying $f(A \sigma_W B)=f(A) \sigma_W f(B)$ is necessarily constant.
\begin{proof}
The case omitting $I_2$-type algebras has been already  dealt with in \cite{ML2}, and we have just proven the statement for $I_2$-type algebras.  
\end{proof}
\end{corollary}
In the previous theorem, when we showed that $U A \sigma_W B=A \sigma_W B U$, we essentially showed that the arithmetic and the Wasserstein mean of pairs of matrices in the form of $sI+t\sigma_x,sI+t\sigma_z,\,(|t|<1)$ commute. This is not true in general. In fact, we have the following result.
\begin{remark}
Let $\cA$ be a $C^*$-algebra such that for every $A,B \in \cA^{++}$ $\frac{A+B}{2} \cdot A \sigma_W B=A \sigma_W B \cdot \frac{A+B}{2}$ holds. Then $\cA$ is commutative. Moreover, if $A \in \cA^{++}$ is an element such that for all $B \in \cA^{++}\,$ $\frac{A+B}{2} \cdot A \sigma_W B=A \sigma_W B \cdot \frac{A+B}{2}$ holds, then $A$ is central. 
\end{remark}
\begin{proof}
It will be beneficial to write the arithmetic mean in its Kubo-Ando form: $A^{\fel}\frac{I+A^{-\fel}BA^{-\fel}}{2}A^{\fel}$.
Then the equation reads as follows:
$$A^{-\fel}\ler{\frac{A+(A^{\fel}BA^{\fel})^{\fel}}{2}}^2 A^{-\fel} \cdot A^{\fel}\frac{I+A^{-\fel}BA^{-\fel}}{2}A^{\fel}=
$$
$$
=A^{\fel}\frac{I+A^{-\fel}BA^{-\fel}}{2}A^{\fel}\cdot A^{-\fel}\ler{\frac{A+(A^{\fel}BA^{\fel})^{\fel}}{2}}^2 A^{-\fel}.$$
The $A^{-\fel},\,A^{\fel}$ terms in the middle cancel out in both sides of the equation. After multiplying both sides with $4$ and $2$ we arrive at:
$$A^{-\fel}\ler{A+(A^{\fel}BA^{\fel})^{\fel}}^2 \cdot (I+A^{-\fel}BA^{-\fel})A^{\fel}=A^{\fel}(I+A^{-\fel}BA^{-\fel})\cdot \ler{A+(A^{\fel}BA^{\fel})^{\fel}}^2 A^{-\fel}.$$
Now after multiplying with $A^{\fel}$ from both sides we get that:
\be\label{eq:ariwas}
\ler{A+(A^{\fel}BA^{\fel})^{\fel}}^2 \cdot (I+A^{-\fel}BA^{-\fel})A=A(I+A^{-\fel}BA^{-\fel}) \cdot \ler{A+(A^{\fel}BA^{\fel})^{\fel}}^2 .
\ee
The left hand side of \eqref{eq:ariwas} equals to:
$$
A^3+(A^{\fel}BA^{\fel})^{\fel}A^2+A(A^{\fel}BA^{\fel})^{\fel}A+A^{\fel}BA^{\frac{3}{2}}+A^{\frac{3}{2}}BA^{\fel}+
$$
$$+(A^{\fel}BA^{\fel})^{\frac{3}{2}}+A(A^{\fel}BA^{\fel})^{\fel}A^{-\fel}BA^{\fel}+A^{\fel}B^2A^{\fel}.
$$
The right hand side equals to:
$$
A^3+A(A^{\fel}BA^{\fel})^{\fel}A+A^2(A^{\fel}BA^{\fel})^{\fel}+A^{\frac{3}{2}}BA^{\fel}+A^{\fel}BA^{\frac{3}{2}}+
$$
$$
+A^{\fel}BA^{-\fel}(A^{\fel}BA^{\fel})^{\fel}A+(A^{\fel}BA^{\fel})^{\frac{3}{2}}+A^{\fel}B^2A^{\fel}.
$$
Most of the terms cancel out and what we get is:
$$(A^{\fel}BA^{\fel})^{\fel}A^2+A(A^{\fel}BA^{\fel})^{\fel}A^{-\fel}BA^{\fel}=A^2(A^{\fel}BA^{\fel})^{\fel}+A^{\fel}BA^{-\fel}(A^{\fel}BA^{\fel})^{\fel}A.$$
This must hold for all $A,B \in \cA^{++}$, so if we replace $B$ with $t^2B$ and take derivative at $t=0$ we get that
$$A^2(A^{\fel}BA^{\fel})^{\fel}=(A^{\fel}BA^{\fel})^{\fel}A^2 \qquad A,B \in \cA^{++}.$$
After multiplying with $A^{-2}$ from the right hand side and then squaring both sides we arrive at
$$A^{\frac{5}{2}}BA^{-\frac{3}{2}}=A^{\fel}BA^{\fel},$$
which is equivalent to $A^2B=BA^2$.
An element commutes with the same elements as its square root, therefore we get that the positive definite cone of $\cA$ is commutative. Since any element of a $C^*$-algebra is the linear combination of at most four positive elements, we get that $\cA$ is commutative.
Note that we showed more: we did not change $A$ in the proof at all, so the statement can be strengthened in the following way: if $A$ is an element in $\cA^{++}$ such that \eqref{eq:ariwas} holds for all $B \in \cA^{++}$, then $A$ is necessarily a central element of the algebra. 
\end{proof}

During the proof of Theorem \ref{thm:main} we showed that the arithmetic and Kubo-Ando $p$-power means of pairs of matrices in the form of $sI+t\sigma_x,\,sI+t\sigma_z, \, (|t|<1)$ commute. We have a similar result to the one that was discussed in the previous Remark.

\begin{remark}
Let $p \in [-1,1)\setminus{\{0\}}$ and let $\cA$ be a $C^*$-algebra such that $A\emp B \cdot \frac{A+B}{2}=\frac{A+B}{2} \cdot A \emp B$ for every $A,B \in \cA^{++}$. Then $\cA$ is commutative. Moreover, if $A \in \cA^{++}$ and $A\emp B \cdot \frac{A+B}{2}=\frac{A+B}{2} \cdot A \emp B$ for every $B \in \cA^{++}$ for some $p \in [-1,1)\setminus{\{0\}}$ then $A$ is central. 
\end{remark}

\begin{proof}
After multiplying with $2^{\frac{1}{p}} \cdot 2$ we are left with
$$\ler{I+(A^{-\fel}BA^{-\fel})^p}^{\frac{1}{p}}A(I+A^{-\fel}BA^{-\fel})=(I+A^{-\fel}BA^{-\fel})A\ler{I+(A^{-\fel}BA^{-\fel})^p}^{\frac{1}{p}}.$$
This must be true for all positive definite elements. If $A,B$ are positive definite elements, then so is $A^{\fel}BA^{\fel}$. Replacing $B$ with $A^{\fel}BA^{\fel}$ gets us to
$$\ler{I+B^p}^{\frac{1}{p}}A(I+B)=(I+B)A\ler{I+B^p}^{\frac{1}{p}}.$$
Now we replace $B$ with $\eps^{\frac{1}{p}}B$ and then we use the binomial series expansion around the identity for small values of $\eps$.
$$\ler{I+\frac{1}{p}\eps B^p+\Oeps}A\ler{I+\eps^{\frac{1}{p}}B}=\ler{I+\eps^{\frac{1}{p}}B}A\ler{I+\frac{1}{p}\eps B^p+\Oeps}.$$
We get that 
\be
\begin{split}
A+\eps^{\frac{1}{p}}AB+\frac{1}{p}\eps B^pA+\frac{1}{p}\eps^{\frac{1}{p}+1}B^pAB+\Oeps= \\
=A+\frac{1}{p}\eps AB^p+\eps^{\frac{1}{p}}BA+\frac{1}{p}\eps^{\frac{1}{p}+1}BAB^p+\Oeps.
\end{split}
\ee
Let us assume first that $0<p<1$. Therefore, $\frac{1}{p}>1$ and it makes sense to differentiate with respect to $\eps$ and then plug in $\eps=0$.
From this, we get that $\frac{1}{p}AB^p=\frac{1}{p}B^pA$. Since the function $t \to t^p$ is a bijection on the positive definite cone, we get that $AB=BA$ for all positive definite $A,B$ which concludes the proof. Note that similarly to the previous remark, we did not change $A$ during the proof, therefore the statement can be strengthened in the following way: if $A$ is a fixed positive definite element in a $C^*$-algebra $\cA$ such that $A\emp B \cdot \frac{A+B}{2}=\frac{A+B}{2} \cdot A \emp B$ for some $p \in (-1,1)\setminus{\{0\}}$ for all $B \in \cA^{++}$, $A$ is necessarily central. 
Now let us turn to the case $-1<p<0$.  The fact that $Am_{-p}B=(A^{-1}m_pB^{-1})^{-1}$ implies that $Am_pB=(A^{-1}m_{-p}B^{-1})^{-1}$, therefore we can assume that $0<p<1$ in this case as well, if we write $A^{-1},B^{-1}$ in the Kubo-Ando power mean and invert it, and $A,B$ in the arithmetic mean.
Once again, we interpret the arithmetic mean as a $p$-power mean for $p=1$. The equation we have is
\be
\begin{split}
A^{\fel}\ler{\frac{I+(A^{\fel}B^{-1}A^{\fel})^p}{2}}^{-\frac{1}{p}}A^{\fel} \cdot A^{\fel} \frac{I+A^{-\fel}BA^{-\fel}}{2}A^{\fel} \\
= A^{\fel} \frac{I+A^{-\fel}BA^{-\fel}}{2}A^{\fel} \cdot A^{\fel}\ler{\frac{I+(A^{\fel}B^{-1}A^{\fel})^p}{2}}^{-\frac{1}{p}}A^{\fel}.
\end{split}
\ee
We can cancel out the $A^{\fel}$ terms from the ends of both sides and after multiplying with $2^{-\frac{1}{p}+1}$ we are left with
$$\ler{I+(A^{\fel}B^{-1}A^{\fel})^p}^{-\frac{1}{p}}A(I+A^{-\fel}BA^{-\fel})=(I+A^{-\fel}BA^{-\fel})A\ler{I+(A^{\fel}B^{-1}A^{\fel})^p}^{-\frac{1}{p}}.$$
Just like in the case $0<p<1$ we can replace $B$ with $A^{\fel}BA^{\fel}$ and arrive at
$$\ler{I+(B^{-1})^p}^{-\frac{1}{p}}A(I+B)=(I+B)A\ler{I+(B^{-1})^p}^{-\frac{1}{p}}.$$
Now if we replace $B$ with $B^{-1}$ and then replace it with $\eps^{\frac{1}{p}}B$ we get that
$$(I+\eps B^p)^{-\frac{1}{p}}A(I+\eps^{-\frac{1}{p}}B^{-1})=(I+\eps^{-\frac{1}{p}}B^{-1})A(I+\eps B^p)^{-\frac{1}{p}}.$$
Now we use binomial expansion around the identity:
$$(I-\frac{1}{p}\eps B^p+\Oeps)A(I+\eps^{-\frac{1}{p}}B^{-1})=(I+\eps^{-\frac{1}{p}}B^{-1})A(I-\frac{1}{p}\eps B^p+\Oeps).$$
This is equivalent to
\be
\begin{split}
A+\eps^{-\frac{1}{p}}AB^{-1}-\frac{1}{p}\eps B^p A-\frac{1}{p}\eps^{1-\frac{1}{p}}B^p AB^{-1}+\Oeps+\eps^{-\frac{1}{p}}AB^{-1}\Oeps = \\
=A-\frac{1}{p}\eps AB^p+\eps^{-\frac{1}{p}}B^{-1}A-\frac{1}{p}\eps^{1-\frac{1}{p}}B^{-1}AB^p+\Oeps+\eps^{-\frac{1}{p}}B^{-1}A \Oeps. 
\end{split}
\ee
$A$ cancels out and after multiplying with $\eps^{\frac{1}{p}}$ we get that
\be
\begin{split}
AB^{-1}-\frac{1}{p}\eps^{1+\frac{1}{p}} B^pA-\frac{1}{p}\eps B^pAB^{-1}+\Oeps \\
=-\frac{1}{p}\eps^{1+\frac{1}{p}} AB^p+B^{-1}A-\frac{1}{p}\eps B^{-1}AB^p+\Oeps.
\end{split}
\ee
Now we differentiate with respect to $\eps$ and then plug in $\eps=0$. We arrive at
$$-\frac{1}{p}B^pAB^{-1}=-\frac{1}{p}B^{-1}AB^p.$$
This means that $AB^{p+1}=B^{p+1}A$. Since $t \to t^{p+1}$ is a bijection on the positive definite cone, we obtain that $A$ commutes with any positive definite element of $\cA$ which concludes the proof.
Proof for the case $p=-1$.
We have that 
\be\label{eq:harmonikus}
\frac{2}{A^{-1}+B^{-1}} \cdot \frac{A+B}{2}=\frac{A+B}{2} \cdot \frac{2}{A^{-1}+B^{-1}}.
\ee
After multiplying from both sides with $(A^{-1}+B^{-1})$ we get
$$(A+B)(A^{-1}+B^{-1})=(A^{-1}+B^{-1})(A+B).$$
This is equivalent to
$$AB^{-1}+BA^{-1}=A^{-1}B+B^{-1}A.$$
Now we replace $B$ with $tB$, and after multiplying with $t$ we obtain that
$$AB^{-1}+t^2BA^{-1}=t^2A^{-1}B+B^{-1}A.$$
We have a polynomial of $t$ on each side of the equation, say $p(t)$ and $q(t)$, and these polynomials are equal for all $t>0$. Since two polynomials are equal if and only if their coefficients coincide, we get that $AB^{-1}=B^{-1}A$, which implies that $A$ commutes with the positive definite cone of $\cA$, therefore $A$ is central. 
\end{proof}

\paragraph*{{\bf Acknowledgement.}}
We are grateful to the anonymous referee for his/her valuable comments and suggestions.

\end{document}